\documentclass[preprint,sort&compress,12pt]{elsarticle}
\usepackage{amscd,amssymb,amsmath,amsthm}
\usepackage[all]{xy}
\usepackage{etoolbox}
\makeatletter
\patchcmd{\ps@pprintTitle}{\footnotesize\itshape
       Preprint submitted to \ifx\@journal\@empty Elsevier
       \else\@journal\fi\hfill\today}{\relax}{}{}
\makeatother

\newdir{ >}{!/8pt/\dir{}*\dir{>}}

\newtheorem{theorem}{Theorem}[section]

\newtheorem{lemma}[theorem]{Lemma}
\newtheorem{corollary}[theorem]{Corollary}
\newtheorem{prop}[theorem]{Proposition}

\newtheorem*{con7*}{Conjecture 7*}
\newtheorem*{Remark}{Remark}
\theoremstyle{definition}
\newtheorem{definition}[theorem]{Definition}

\newtheorem*{th31}{Theorem 3.1}
\newtheorem*{th49}{Theorem 4.9}
\newtheorem*{th43}{Theorem 4.3}
\newtheorem*{th45}{Theorem 4.5}

\theoremstyle{remark}

\numberwithin{equation}{theorem}

\journal{}

\begin{document}

\begin{frontmatter}

\title{ On the second stable homotopy group of the Eilenberg-Maclane space}

 \author[IISER TVM]{A. Antony}
\ead{ammu13@iisertvm.ac.in}
\author[IISER TVM]{G. Donadze}
\ead{gdonad@gmail.com}
\author[IISER TVM]{V. Prasad}
\ead{vishnuprasad@iisertvm.ac.in}
\author[IISER TVM]{V.Z. Thomas\corref{cor1}}
\address[IISER TVM]{School of Mathematics,  Indian Institute of Science Education and Research Thiruvananthapuram,\\695016
Kerala, India}
\ead{vthomas@iisertvm.ac.in}
\cortext[cor1]{Corresponding author. \emph{Phone number}: +91 9744041060.}

\begin{abstract}
We prove that for a finitely generated group $G$, the second stable homotopy group $\pi_2^S(K(G,1))$ of the Eilenberg-Maclane space $K(G,1)$ is completely determined by the Schur multiplier $H_2(G)$. We also prove that the second stable homotopy group $\pi_2^S(K(G,1))$ is equal to the Schur multiplier $H_2(G)$ for a torsion group $G$ with no elements of order $2$ and show that for such groups, $\pi_2^S(K(G,1))$ is a direct factor of $\pi_{3}(SK(G,1))$, where $S$ denotes suspension and $\pi_2^S$ the second stable homotopy group. We compute $\pi_{3}(SK(G,1))$ and $\pi_2^S(K(G,1))$ for symmetric, alternating, dihedral,  general linear groups over finite fields and some infinite general linear groups $G$. We also obtain a bound for $\pi_2^S(K(G,1))$ for of all finite groups $G$.
\end{abstract}


\end{frontmatter}

\section{Introduction}

One of the aims of this paper is to compute the second stable homotopy group of  $K(G,1)$ and to compute $\pi_3(SK(G,1))$.
For this purpose, first we obtain some structural results for the non-abelian tensor square of groups. R. Brown and J.-L. Loday introduced the nonabelian tensor product $G\otimes H$ for a pair of groups $G$ and $H$ in \cite{BL1} and \cite{BL2} in the context of an application in homotopy theory, extending the ideas of J.H.C. Whitehead in \cite{W}. The major contribution of the paper \cite{BL2} is to prove a higher homotopy van Kampen theorem, which gives the transition from topology to algebra. More information can be found in \cite{BHR}. The further contribution was to define and apply the nonabelian tensor product. A special case, the nonabelian tensor square, already appeared in the work of R.K. Dennis in \cite{RKD}. The non-abelian tensor product of groups is defined for a pair of groups that act on each other provided the actions satisfy the compatibility conditions of Definition \ref{D:1.1} below. Note that we write conjugation on the left, so $^gg'=gg'g^{-1}$ for $g,g'\in G$ and
$^gg'\cdot g'^{-1}=[g,g']$ for the commutator of $g$ and $g'$.

\begin{definition}\label{D:1.1}
Let $G$ and $H$ be groups that act on themselves by conjugation and each of which acts on the other. The mutual actions are said to be compatible if
\begin{equation}
\tag{1.1.1}   ^{^h g}h'=\; ^{hgh^{-1}}h' \;and\; ^{^g h}g'=\ ^{ghg^{-1}}g' \;\mbox{for \;all}\; g,g'\in G, h,h'\in H.
\end{equation}
\end{definition}

\begin{definition}\label{D:1.2}
If $G$ and $H$ are groups that act compatibly on each other, then the nonabelian tensor product $G\otimes H$ is the group generated by the symbols $g\otimes h$ for $g\in G$ and $h\in H$ with relations
\begin{equation}\label{E:1.2.1}
gg'\otimes h=(^gg'\otimes \;^gh)(g\otimes h),    \tag{1.2.1}
\end{equation}
\begin{equation}\label{E:1.2.2}
g\otimes hh'=(g\otimes h)(^hg\otimes \;^hh'),    \tag{1.2.2}
\end{equation}
\noindent for all $g,g'\in G$ and $h,h'\in H$.
\end{definition}

The special case where $G=H$, and all actions are given by conjugation, is called the tensor square $G\otimes G$. The tensor square of a group is always defined.

There exists a homomorphism $\kappa : G\otimes G \rightarrow G^{\prime}$ sending $g\otimes h$ to $[g,h]$. Set $J(G)=\ker(\kappa)$. Its topological interest is the formula $J(G)\cong \pi_{3}(SK(G,1))$, where $SK(G,1)$ is the suspension of $K(G,1)$. The group $J(G)$ lies in the centre of $G\otimes G$.

Let $\Delta(G)$ denote the subgroup of $J(G)$ generated by the elements \newline $(x\otimes y)(y\otimes x)$ for $x,y\in G$. The symmetric product of $G$ is then defined as $G\tilde{\otimes} G = (G\otimes G)/\Delta(G)$. We set $\tilde{J}(G)= J(G)/\Delta(G)$. It is shown in \cite{BL2} that $\tilde{J}(G)\cong \pi_{4}(S^{2}K(G,1))= \pi^{s}_{2}(K(G,1))$.

Let $\nabla (G)$ denote the subgroup of $J(G)$ generated by the elements $x\otimes x$ for $x\in G$. The exterior square of $G$ is defined as $G\wedge G= (G\otimes G)/\nabla (G)$. We set $J(G)/\nabla(G)=M(G)$, which is otherwise known as the Schur multiplier of $G$. It has been shown in \cite{M} that $M(G)\cong H_{2}(G)$, the second homology group of $G$.
The importance and relation of nonabelian tensor product to other constructions can be captured in the following commutative diagram given in \cite{BL2}

\begin{equation*}\label{D:1}
\xymatrix@+20pt{
0\ \ar@{->}[r]
&\pi_{3}(SK(G,1))\ar@{->}[r]
\ar@{->}[d]
 &G\otimes G\ar@{->}[r]
\ar@{->}[d]
&G'\ar@{->}[r]
\ar@{->}[d]
&1\hspace{12mm} \\
0\ \ar@{->}[r]
&\pi_{2}^{S}(K(G,1))\ar@{->}[r]
\ar@{->}[d]
 &G\widetilde{\otimes} G\ar@{->}[r]
\ar@{->}[d]
&G'\ar@{->}[r]
\ar@{->}[d]
&1\hspace{8mm}(\ref{D:1}) \\
0\ \ar@{->}[r]
&H_2(G)\ar@{->}[r]
 &G\wedge G\ar@{->}[r]
&G'\ar@{->}[r]
&1\hspace{12mm} \\
}\end{equation*}

Now we will briefly describe how the paper is organised. In section two, we begin with some preparatory results and prove that when $G'$ has a complement in $G$, then $\Delta(G)\cong \Delta(G_{ab})$ and $\nabla(G)\cong \nabla(G_{ab})$.

The authors of \cite{BJR} give a formula for $G\otimes G$ when $G=H\times K$ in terms of the factors $H$ and $K$. In section 3, we want to do the same when $G=N\rtimes H$. We also express $G\wedge G$ as a semidirect product, and we express $\nabla(G)$ and $\Delta(G)$ as a direct product.
Our contribution is the following
\begin{th31}
Let $G=N\rtimes H$, then the following statements hold:
\begin{itemize}
 \item[(i)] $G \otimes G \cong K_1 \rtimes (H\otimes H) $, where $K_1$ is the normal subgroup of $G\otimes G$ generated by $\{g\otimes n, n_1\otimes g_1|g, g_1\in G, n, n_1\in N\}$.
  \item[(ii)] $G \wedge G \cong K_2 \rtimes (H \wedge H)$, where $K_2$ is the normal subgroup of $G\wedge G$ generated by $\{g\wedge n|g\in G, n\in N\}$.
 \item[(iii)] $\nabla (G) \cong K_3 \times \nabla(H)$, where $K_3$ is the normal subgroup of $\nabla (G)$ generated by $ \{ (g\otimes n)(n\otimes g),(n_1\otimes n_1) | g \in G, n, n_1 \in N \}$.
\end{itemize}
\end{th31}

The authors of \cite{MW} study homotopy groups of the suspensions of classifying spaces of groups. In section 4, we prove that the second stable homotopy group of $K(G,1)$ is completely determined by the Schur multiplier for a finitely generated group $G$. We also provide a direct product decomposition for third homotopy group of the suspension of $K(G,1)$ with the Schur multiplier as one of the factors. If $G_{ab}$ is finitely generated, then the authors of \cite{BFM} prove that $G\otimes G\cong \nabla(G)\times (G\wedge G)$, when $G'$ has a complement in $G$ and also when $G_{ab}$ has no elements of order 2.  Our aim in this section is to generalize this result. We obtain the above direct product description by constructing an explicit splitting map. In particular, we obtain the following results
\begin{th43}
If $G$ is finitely generated group, then $G\ \widetilde{\otimes}\ G\cong (G\wedge G) \times \nabla(G)/{\Delta(G)}$. In particular,  $\pi_2^{S}(K(G,1))\cong H_2(G)\times (\mathbb{Z}/2\mathbb{Z})^{r+k}$, where $r$ is the rank of $G_{ab}$ and $k$ is the number of cyclic groups of even order in the decomposition of $G_{ab}$
\end{th43}

\begin{th45}
Let $k$ and $n$ be natural numbers. If $(g\otimes g)^k=1$ for all $g\in G$ and for any odd integer $k$ or there exist an even integer $n$ such that $G$ has $n^{th}$ roots, then in both cases $\pi_{2}^S(K(G,1))\cong H_2 (G)$.
\end{th45}

\begin{th49}
 If $G$ is a torsion group with no elements of order $2$, then $G\otimes G \cong \nabla(G)\times (G\wedge G)$. In particular, $\pi_2^S(K(G,1))\cong H_2(G)$ and $\pi_3(SK(G,1))\cong H_2(G)\times \nabla(G)$. 
\end{th49}

In section 5, we give some applications of the results proved in this paper. In particular, we compute $\pi_3(SK(G,1))$ and $\pi_2^S(K(G,1))$ for symmetric groups, alternating groups, general linear groups over finite fields and infinite general linear groups.

\section{Preparatory Results}

The following proposition can be found in \cite{BJR}, we record it here for easy access and because we use it extensively.

\begin{prop}
Let $G$ be a group, then the following hold:
\begin{itemize}
\item[(i)]\label{L:2.1}
$G$ acts trivially on $J(G)$. In particular $G$ acts trivially on both $\nabla (G)$ and $\Delta (G)$.

\item[(ii)]\label{L:2.2}
 $J(G)$ is contained in the centre of $ G \otimes G$.

\item[(iii)]\label{L:2.3}
Let $G$ be a group. We have $ \Delta (G) \subseteq \nabla (G)$.

\item[(iv)]\label{L:2.4}
Let $G$ be a group. Then $x\otimes x = 1$ if $x \in G'$.

\end{itemize}
\end{prop}

 \begin{lemma}\label{L:2.5}
  Let $G$ be a group. Then $(x\otimes a)(a\otimes x)=1$ for all $x\in G'$ and $a\in G$.
 \end{lemma}
 \begin{proof}
 In light of Proposition 2.3(d) of \cite{BL2}, it suffices to prove the result for $x=[g,h]$. We have
  \begin{align*}
   &(ghg^{-1}h^{-1} \otimes a)(a \otimes ghg^{-1}h^{-1})\\
    &= (g\otimes h)\;^{a}(g\otimes h)^{-1}\;^{a}(g\otimes h)(g\otimes h)^{-1}\\
    &=1.
  \end{align*}
\end{proof}

\begin{prop}\label{P:2.4}
Let $G$ be a group and set $A=G_{ab}$. If $G'$ has a complement in $G$, then the following hold.
\begin{itemize}
\item[(i)]  $\nabla (G) \cong \nabla (A)$.
\item[(ii)] $\Delta (G) \cong \Delta (A)$
\end{itemize}
\end{prop}

\begin{proof}
We have $G = G'B$ where $B \cong A$. Identifying $A$ with its isomorphic copy in $G$, we can write every $g\in G$ as $g=xy$ where $x\in G'$ and $y \in A$.
   
 \begin{itemize}
   \item[(i)]Let $g \in G$. Then $g = xa$, where $x \in G^\prime$ and $a \in A$. Now we have
  \begin{align*}
   g \otimes g
   &=\;^x(a \otimes x)\;^{x^2}(a\otimes a)(x\otimes x)^x(x\otimes a)\hspace{4mm}[(\ref{E:1.2.1})\;\mbox{and}\; (\ref{E:1.2.2})] \\
   &=(x \otimes x)(a \otimes a)(a \otimes x)(x \otimes a),\hspace{6mm}\mbox{[Proposition \ref{L:2.1}]} \\
   &= (a \otimes a).
  \end{align*}
Hence the result follows

  \item[(ii)]
 Let $g = xa, h = yb, x, y\in G', a,b\in A$.
 Using (\ref{E:1.2.1}) and (\ref{E:1.2.2}) we have,

 \begin{equation*}\label{E:2.4.1}
 (xa \otimes yb) = {^{x}(a\otimes y)} {^{xy}(a\otimes b)}(x\otimes y) {^{y}(x\otimes b)}  \tag{2.4.1}
 \end{equation*}
 \begin{equation*}\label{E:2.4.2}
 (yb \otimes xa) = {^{y}(b\otimes x)} {^{yx}(b\otimes a)}(y\otimes x) {^{x}(y\otimes a)}   \tag{2.4.2}
 \end{equation*}

From Lemma\;(\ref{L:2.5}) we get
\begin{equation*}\label{E:2.4.3}
(x\otimes b)(b\otimes x) = 1 \tag{2.4.3}
\end{equation*}

Using (\ref{E:2.4.1}),\;(\ref{E:2.4.2}) and (\ref{E:2.4.3})  we obtain,
\begin{align*}
 &(xa \otimes yb)(yb \otimes xa)\\
&= {^{x}(a\otimes y)} {^{xy}(a\otimes b)}(x\otimes y)  {^{yx}(b\otimes a)}(y\otimes x) {^{x}(y\otimes a)}\\
&= {^{x}(a\otimes y)} {^{xy}(a\otimes b)}  {^{xy}(b\otimes a)} (x\otimes y)(y\otimes x) {^{x}(y\otimes a)}\\
&= (a\otimes y)(y\otimes a)(a\otimes b)(b\otimes a)(x\otimes y)(y\otimes x)\hspace{8mm} \mbox{[Proposition \ref{L:2.1}]}\\
&=(a\otimes b)(b\otimes a).
\end{align*}
Hence the result follows.
  \end{itemize}
\end{proof}

\section{Tensor square of semi-direct product of groups}\label{S:3}
 In \cite{BJR}, the authors give a description of the tensor square of a direct product of groups. In this section we do the same for semi-direct product.
\begin{theorem}\label{T:3.1}
 Let $G=N\rtimes H$, then the following statements hold:
\begin{itemize}
 \item[(i)] $G \otimes G \cong K_1 \rtimes (H\otimes H) $, where $K_1$ is the normal subgroup of $G\otimes G$ generated by $\{g\otimes n, n_1\otimes g_1|g, g_1\in G, n, n_1\in N\}$.
  \item[(ii)] $G \wedge G \cong K_2 \rtimes (H \wedge H)$, where $K_2$ is the normal subgroup of $G\wedge G$ generated by $\{g\wedge n|g\in G, n\in N\}$.
 \item[(iii)] $\nabla G \cong K_3 \times \nabla(H)$, where $K_3 = < (g\otimes n)(n\otimes g),(n_1\otimes n_1) | g \in G, n, n_1 \in N >$.
 \item[(iv)] $\Delta G \cong K_4 \times \Delta (H)$, where $K_4 = \ker (p\otimes p)|_{\Delta G}$ .
 \item[(v)] $G \widetilde{\otimes} G \cong K_5 \rtimes (H \widetilde{\otimes} H)$ where, $K_5 =  \;< (g\otimes n)\Delta G | g \in G, n \in N >$
\end{itemize}
 \end{theorem}
 \begin{proof}
 \begin{itemize}
\item[(i)]
 Consider the short exact sequence of groups
 \begin{equation*}
\xymatrix@+20pt{
1\ \ar@{->}[r]
& N \ar@{->}[r]
 &G \ar@{->}[r]^{p}
&H \ar@{->}[r]
&1
,
}\end{equation*}
where $p$ is the natural projection. By hypothesis, there exists a map $\alpha : H \to G$ such that $p \circ \alpha = 1_H$. First we claim that
\begin{equation*}
\xymatrix@+20pt{
1\ \ar@{->}[r]
& K_1 \ar@{->}[r]
 &G\otimes G\ar@{->}[r]^{p\otimes p}
&H \otimes H\ar@{->}[r]
&1
,
}\end{equation*}
is an exact sequence. Clearly $K_1 \subseteq \ker(p\otimes p)$, and hence obtain an induced map $p\otimes p : G\otimes G/K_1 \longrightarrow H\otimes H$ defined by $(p\otimes p)((g\otimes h)K_1) = (p(g)\otimes p(h))$. Now we will define an inverse map $f :H\otimes H \longrightarrow (G\otimes G)/K_1$ given by $f(p(g_1) \otimes p(g_2)) = (g_1\otimes g_2)K_1$. Note that if $p(g_i) = p(g_i')$, $i = 1,2$ then  $g_i' = g_{i}n_i$ where $n _i \in N$.
Expansion using \ref{E:1.2.1} and \ref{E:1.2.2} gives,
\begin{align*}
f(p(g_1')\otimes p(g_2'))&= (g_1 n_1 \otimes g_2 n_2)K_1\\
 &=\ ^{g_1}(n_1\otimes g_2)^{g_1g_2}(n_1\otimes n_2)(g_1\otimes g_2)^{g_2}(g_1\otimes n_2)K_1\\
&= (g_1\otimes g_2)K_1
\end{align*}
Observe that $(p\otimes p)\circ f = 1_{H\otimes H}$ and $f\circ (p\otimes p) = 1_{(G\otimes G)/K_1}$. Hence $K_1 = \ker(p\otimes p)$.

 Now we prove that this exact sequence splits on the right and gives a semi-direct product.
 Consider the homomorphism $\alpha \otimes \alpha : H \otimes H \longrightarrow G \otimes G$ defined by $(\alpha \otimes \alpha)(a \otimes b) = \alpha(a) \otimes \alpha(b)$. Note that $(p \otimes p)\circ(\alpha \otimes \alpha) (h_1 \otimes h_2) = (p \circ \alpha)(h_1) \otimes (p \circ \alpha)(h_2) = h_1 \otimes h_2$. Thus the above exact sequence splits, and we have  $G \otimes G \cong K_1 \rtimes (H \otimes H) $.

\item[(ii)]  As in (i), we have an exact sequence

\begin{equation*}
\xymatrix@+20pt{
1\ar@{->}[r]
&K_2\ar@{->}[r]
&G\wedge G\ar@{->}[r]^{p'\otimes p'}
&H \wedge H\ar@{->}[r]
&1
.
}\end{equation*}

where $(p' \otimes p')(a \wedge b) = (p(a) \wedge p(b))$.
Just as above we obtain the required semi-direct product using this short exact sequence.

 \item[(iii)]Considering the restriction of $p \otimes p$ to $\nabla (G)$  and proceeding as in (i) we get the following exact sequence

\begin{equation*}
\xymatrix@+20pt{
1\ \ar@{->}[r]
& K_3 \ar@{->}[r]
 &\nabla (G) \ar@{->}[r]^{p\otimes p}
&\nabla (H) \ar@{->}[r]
&1
.
}\end{equation*}
 which splits on the right and we get a direct product because $\nabla (G)$ is an abelian group.

\item[(iv)] Further, restricting $p\otimes p$ to $\Delta (G)$ we get the following short exact sequence:
\begin{equation*}
\xymatrix@+20pt{
1\ \ar@{->}[r]
&\ker(p\otimes p)|_{\Delta (G)}\ar@{->}[r]
 &\Delta (G) \ar@{->}[r]^{p\otimes p}
&\Delta (H) \ar@{->}[r]
&1
.
}\end{equation*}
and as in the previous cases we can get this sequence to split on the right and there by giving the required direct product.

\item[(v)] Again let us consider the following short exact sequence:
   \begin{equation*}
\xymatrix@+20pt{
1\ \ar@{->}[r]
&\ker(p''\otimes p'')\ar@{->}[r]
 &G\;\widetilde{\otimes}\; G\ar@{->}[r]^{(p''\otimes p'')}
&H \;\widetilde{\otimes}\; H\ar@{->}[r]
&1
.
}\end{equation*}

  where $(p'' \otimes p'')(a \otimes b)\Delta (G) = (p(a) \otimes p(b))\Delta (G)$. Proceeding as in (i), we obtain the result.

 \end{itemize}
 \end{proof}

\begin{prop}\label{T:3.2}
Let $G$ be a group and $N\unlhd G$ be  perfect. If $N$ has a compliment, $G\otimes G \cong j(N\otimes N)\rtimes (G/N\otimes G/N)$, where $j$ is the  natural map from $N\otimes N$ to $G\otimes G$.
\end{prop}
\begin{proof}
By Theorem \ref{T:3.1}, we have the following exact sequence which splits on the right.
\begin{equation*}
1\longrightarrow K_1\longrightarrow G\otimes G \longrightarrow G/N\otimes G/N\longrightarrow 1
\end{equation*}
 Let $n\in N,\ g\in G$. Since $N$ is perfect, $n = \prod_{i}[a_i,b_i],$ for some $ a_i, b_i\in N$. We prove for $n = [a_1,b_1]$ and the general case follows easily by induction. Now $n\otimes g = [a_1,b_1]\otimes g = (a_1\otimes b_1) ^{g}(a_1\otimes b_1)^{-1}$.  Hence we have $K_1 = j(N\otimes N)$. Thus we have,
\begin{equation*}
\xymatrix@+20pt{
1\ \ar@{->}[r]
&j(N\otimes N)\ar@{->}[r]
 & G \otimes G\ar@{->}[r]^{(p\otimes p)}
&G/N \otimes G/N \ar@{->}[r]
&1
}\end{equation*}
which splits on the right and our result follows.
\end{proof}

As an easy consequence of our discussion, we obtain the following proposition which is used later for explicit computations.

\begin{prop} \label{P:3.3}
Let $G$ be a group. Let $N \trianglelefteq G$ be perfect. Then the following holds :
\begin{itemize}
\item[(1)] $\pi_3(SK(N,1)) \longrightarrow \pi_3(SK(G,1)) \longrightarrow \pi_3(SK(G/N,1))\longrightarrow 0.$
\item[(2)] $H_2(N) \longrightarrow H_2(G) \longrightarrow H_2(G/N)  \longrightarrow 0.$
\item[(3)] $\pi^s_2(K(N,1))  \longrightarrow \pi^s_2(K(G,1)) \longrightarrow \pi^s_2(K(G/N,1)) \longrightarrow 0.$
\end{itemize}
\end{prop}
\begin{proof}

Since $N\trianglelefteq G$ is perfect, a consequence of Proposition \ref{T:3.2} yields the following exact sequence of groups
\begin{equation*}
N\otimes N \longrightarrow G\otimes G \longrightarrow G/N\otimes G/N \longrightarrow 0
\end{equation*}
Note that $(G/N)' = G'/N$ and hence we have an exact sequence of groups $0 \to N \to G' \to G'/N \to 0$.

\begin{itemize}
\item[(1)] We have the following commutative diagram

\begin{equation*}
\xymatrix@+20pt{
&N\otimes N\ \ar@{->}[r]
\ar@{->}[d]
 &G\otimes G\ar@{->}[r]
\ar@{->}[d]
&G/N\otimes G/N\ar@{->}[r]
\ar@{->}[d]
&0 \\
0\ \ar@{->}[r]
&N'\ \ar@{->}[r]
 &G'\ar@{->}[r]
&G'/N \ar@{->}[r]
&0
}\end{equation*}
Note that Snake lemma holds true in the category of groups if the vertical maps are surjective. Thus we have,
\begin{equation*}
 \pi_3(SK(N,1)) \longrightarrow \pi_3(SK(G,1)) \longrightarrow \pi_3(SK(G/N,1))\longrightarrow 0.
\end{equation*}

\item[(2)] From Theorem \ref{T:3.1} we have the following exact sequence,
\begin{equation*}
\xymatrix@+20pt{
1\ar@{->}[r]
&K_2\ar@{->}[r]
&G\wedge G\ar@{->}[r]
&G/N \wedge G/N\ar@{->}[r]
&1
.
}\end{equation*}
As $N$ is perfect, just as in Proposition \ref{T:3.2} we have $K_2 = i(N\wedge N)$ where $i(N\wedge N)$ is the image of $N\wedge N$ in $G\wedge G$ under the natural map $i$ from $ N\wedge N$ to $G\wedge G$. So we have the exact sequence,
\begin{equation*}
\xymatrix@+20pt{
&N\wedge N\ \ar@{->}[r]
 &G\wedge G\ar@{->}[r]
&G/N\wedge G/N\ar@{->}[r]
&0}
\end{equation*}

Then we obtain the following commutative diagram.

\begin{equation*}
\xymatrix@+20pt{
&N\wedge N\ \ar@{->}[r]
\ar@{->}[d]
 &G\wedge G\ar@{->}[r]
\ar@{->}[d]
&G/N\wedge G/N\ar@{->}[r]
\ar@{->}[d]
&0 \\
0\ \ar@{->}[r]
&N'\ \ar@{->}[r]
 &G'\ar@{->}[r]
&G'/N\ar@{->}[r]
&0 .
}\end{equation*}
 Now applying Snake lemma yields
\begin{equation*}
H_2(N) \longrightarrow H_2(G) \longrightarrow H_2(G/N)  \longrightarrow 0.
\end{equation*}

\item[(3)] As in the above case we have the image of $N\widetilde{\otimes} N$  in $G\widetilde{\otimes} G$ under the natural map to be equal to $K_5$ which is as defined in Theorem \ref{T:3.1}. We thus obtain the following exact sequence.
\begin{equation*}
\xymatrix@+20pt{
&N\widetilde{\otimes} N\ \ar@{->}[r]
 &G\widetilde{\otimes} G\ar@{->}[r]
&G/N\widetilde{\otimes} G/N\ar@{->}[r]
&0
}\end{equation*}

And we have,

\begin{equation*}
\xymatrix@+20pt{
&N\widetilde{\otimes} N\ \ar@{->}[r]
\ar@{->}[d]
 &G\widetilde{\otimes} G\ar@{->}[r]
\ar@{->}[d]
&G/N\widetilde{\otimes} G/N\ar@{->}[r]
\ar@{->}[d]
&0 \\
0\ \ar@{->}[r]
&N'\ \ar@{->}[r]
 &G'\ar@{->}[r]
&G'/N \ar@{->}[r]
&0 .
}\end{equation*}
Now applying Snake lemma we obtain,
\begin{equation*}
\pi^s_2(K(N,1))  \longrightarrow \pi^s_2(K(G,1)) \longrightarrow \pi^s_2(K(G/N,1)) \longrightarrow 0
\end{equation*}
\end{itemize}
\end{proof}

\section{Second Stable homotopy group of the Eilenberg-Maclane space}

One of the main aims of this section is to prove that $\pi_2^S(K(G,1))$ is completely determined by the Schur multiplier $H_2(G)$. We also prove that $\pi_3(SK(G,1))\cong H_2(G)\times \nabla (G)$ under some hypothesis (cf. Corollary \ref{C:4.7}, Theorem \ref{T:4.9}, Theorem \ref{T:4.8}). For this we show that the short exact sequence $1\to \nabla(G) \to G\otimes G\to G\wedge G\to 1\ (*)$ splits. In order to do this, we first obtain a splitting of the short exact sequence $1\to \nabla(G_{ab}) \to G_{ab}\otimes G_{ab}\to G_{ab}\wedge G_{ab}\to 1$ and then we prove that $\nabla (G) \cong \nabla(G_{ab})$. Now the left square in the diagram below commutes, giving the splitting for the short exact sequence $(*)$.
\begin{equation*}
\xymatrix@+20pt{
1\ \ar@{->}[r]
&\nabla (G)\ar@{->}[r]
\ar@{}^{||\wr}[d]
 &G\otimes G\ar@{->}[r]
\ar@{->}[d]
&G\wedge G \ar@{->}[r]
\ar@{->}[d]
&1 \\
1\ \ar@{->}[r]
&\nabla (G_{ab})\ar@{->}[r]
 &G_{ab}\otimes G_{ab}\ar@{->}[r]
&G_{ab}\wedge G_{ab}\ar@{->}[r]
&1\ 
}\end{equation*}

In some cases (cf. Theorem \ref{T:4.8}) we obtain a direct splitting for the short exact sequence $(*)$.
We begin with the following easy lemma.

\begin{lemma}\label{L:4.1}
Let $A$ be an abelian group. If $A$ is finitely generated, then the following holds
\begin{itemize}
\item[(i)] $ A\otimes A \cong \nabla (A) \times (A \wedge A)$.
\item[(ii)] Let $N$ be a normal subgroup of $A\otimes A$. If $N\subseteq \nabla (A)$, then $(A\otimes A)/N\cong (\nabla (A))/N\times (A\wedge A)$.
\end{itemize}
\end{lemma}

\begin{proof}
\begin{itemize}
\item[(i)]
Consider the short exact sequence
\begin{equation*}
\xymatrix@+20pt{
1\ \ar@{->}[r]
&\nabla (A)\ar@{->}[r]^{i}
 &A\otimes A\ar@{->}[r]^{p}
&A\wedge A\ar@{->}[r]
&1.
}\end{equation*}
where $i$ is the natural inclusion and $p$ is the natural projection. Let $\{x_i, 1\leq i\leq n\}$ be the generators of the cyclic groups in the  decomposition of $A$ as direct product of cyclic groups. Note that $\nabla (A) $ is the normal subgroup of $A\otimes A$ generated by the elements of the form $\{(x_i \otimes x_i), (x_i \otimes x_j)(x_j \otimes x_i),1\leq i <j\leq n\}$. Since $A\otimes A$ decomposes as a direct sum of cyclic groups $\mathbb{Z}_i\otimes \mathbb{Z}_j$, it suffices to define a map $\alpha : A \otimes A \rightarrow \nabla (A)$ on the generators $x_i\otimes x_j$ given by
\begin{align*}
 &\alpha (x_i \otimes x_i) = (x_i \otimes x_i)\\
 &\alpha (x_i \otimes x_j) = (x_i \otimes x_j)(x_j \otimes x_i)\ \mbox{if }\  i < j \ \mbox{and,}\\
 &\alpha (x_i \otimes x_j) = 1\ \mbox{if} \ i > j
\end{align*}
It is easily verified that the above map is a well defined homomorphism.
Notice that  $\alpha ( i(x_i \otimes x_i)) = \alpha(x_i \otimes x_i)= (x_i \otimes x_i)$, and for $i <j$
\begin{align*}
 \alpha i((x_i \otimes x_j)(x_j \otimes x_i)) &= \alpha((x_i \otimes x_j)(x_j \otimes x_i)) \\ &= \alpha(x_i \otimes x_j)\alpha(x_j \otimes x_i) \\ &= (x_i \otimes x_j)(x_j \otimes x_i)
\end{align*}
 Thus $\alpha i = 1$ and the above short exact sequence splits, giving us the desired result.
 
\item[(ii)]  Consider the short exact sequence
\begin{equation*}
\xymatrix@+20pt{
1\ \ar@{->}[r]
&(\nabla (A))/N\ar@{->}[r]^{j}
 &(A\otimes A)/N\ar@{->}[r]^{p_1}
&A\wedge A\ar@{->}[r]
&1.
}\end{equation*}
where $p_1$ is the natural projection and $j$ is the  inclusion map. Our aim is to prove that this sequence splits. As a consequence of (i), we have a well defined homomorphism
$\beta:A\wedge A \longrightarrow A\otimes A$ such that $p \circ \beta = 1_{A\wedge A}$.
Let $p_2:A\otimes A\to (A\otimes A)/N$ be the natural projection. Since $p_1\circ p_2\circ\beta=p\circ\beta=1_{A\wedge A}$, it follows that $p_2\circ\beta: A\wedge A\to (A\otimes A)/N$ is the required splitting.
\end{itemize}
\end{proof}
Taking $\Delta (A) $ as $N$ in the above Lemma we obtain,

\begin{corollary}\label{C:4.2}
$A\ \widetilde{\otimes}\ A \cong \nabla (A) / \Delta (A) \times (A\wedge A)$.
\end{corollary}

Now we come to one of the main results of this section.

\begin{theorem}\label{T:4.2}
If $G$ is finitely generated group, then $G\ \widetilde{\otimes}\ G\cong (G\wedge G) \times \nabla(G)/{\Delta(G)}$. In particular,  $\pi_2^{S}(K(G,1))\cong H_2(G)\times (\mathbb{Z}/2\mathbb{Z})^{r+k}$, where $r$ is the rank of $G_{ab}$ and $k$ is the number of cyclic groups of even order in the decomposition of $G_{ab}$
\end{theorem}
\begin{proof}
Set $\nabla(G)/{\Delta(G)} = \nabla' (G) $ and $\nabla(G_{ab})/{\Delta(G_{ab})} = \nabla' (G_{ab})$.
Let us first prove the theorem under the assumption that the natural projection $\alpha: \nabla' (G) \to \nabla' (G_{ab})$ is an isomorphism. Consider the following commutative diagram
\begin{equation*}
\xymatrix@+20pt{
1\ \ar@{->}[r]
&\nabla'(G)\ar@{->}^{j}[r]
\ar@{}^{||\wr\ \alpha}[d]
 &G\ \widetilde{\otimes}\ G\ar@{->}[r]
\ar@{->}^p[d]
&G\wedge G \ar@{->}[r]
\ar@{->}[d]
&1 \\
1\ \ar@{->}[r]
&\nabla'(G_{ab})\ar@{->}^{i}[r]
 &G_{ab}\ \widetilde{\otimes}\ G_{ab}\ar@{->}[r]
&G_{ab}\ \widetilde{\wedge}\ G_{ab}\ar@{->}[r]
&1.
}\end{equation*}
Our aim is to show that the top row of the above diagram splits. By Corollary \ref{C:4.2}, the bottom row splits and let
  $f: G_{ab}\ \widetilde{\otimes}\ G_{ab} \to \nabla'(G_{ab})$  be the splitting map.  Set $f'=\alpha^{-1} \circ f\circ p$. It is easily checked that $f'$ is the required splitting. Now we proceed to prove that $\alpha$ is an isomorphism. For this, it suffices to show that $|\nabla' (G)|\leq |\nabla' (G_{ab})|$. Suppose that 
$G_{ab}= \mathbb{Z}^r \times \mathbb{Z}/n_1\mathbb{Z} \times \mathbb{Z}/n_2\mathbb{Z} \times \cdots \times \mathbb{Z}/n_m\mathbb{Z}$. Without
loss of generality assume that $n_1, \cdots, n_k$ are even numbers and $n_{k+1}, \cdots, n_m$ are odd numbers. Then
$\nabla' (G_{ab})=(\mathbb{Z}/2 \mathbb{Z})^{r+k}$. Hence it is enough to show that $|\nabla' (G)|\leq 2^{r+k}$. Let $s_1, \cdots, s_r\in G_{ab}$ be
the generators for $\mathbb{Z}^r$ and let $t_i\in G_{ab}$ be the generator of $\mathbb{Z}/n_i\mathbb{Z}$. 
Choose $h_1, \cdots, h_r, g_1, \cdots, g_m\in G$ such that $\tau (h_i)=s_i$ and $\tau (g_i)=t_i$, where $\tau : G\to G_{ab}$ is the
natural projection. Using the following relations
$$
gx\otimes gx = (x\otimes x) (g\otimes g) (^gx\otimes g) (g\otimes \;^gx), \;\; x,g\in G,
$$
$$
y\otimes y =1, \;\; y\in G',
$$
it is easy to see that $\nabla' (G)$ is generated by the images of $h_1\otimes h_1, \cdots, h_r\otimes h_r, g_1\otimes g_1, \cdots, g_m\otimes g_m$. 
Note that $g_i^{n_i}\in G'$ and $(g_i\otimes g_i)^{n_i^2}=g_i^{n_i}\otimes g_i^{n_i}=1_{\otimes}$. This implies that $g_i\otimes g_i$ has odd order for $i>k$.
On the other hand each nontrivial element in $\nabla' (G)$ has order $2$.  Hence $g_i\otimes g_i$ is trivial in $\nabla' (G)$
for $i>k$. Since $\nabla' (G)$ is abelian and each nontrivial element in $\nabla' (G)$ has order $2$, 
we obtain that $|\nabla' (G)|\leq 2^{r+k} \leq |\nabla' (G_{ab})|$.
\end{proof}

In case $G$ is not finitely generated, the next theorem helps us to determine when the Schur multiplier is equal to $\pi_2^S(K(G,1))$. But first we need the following definition.

\begin{definition}
 A group $G$ is said to have $n^{th}$ roots if the map $f : G \longrightarrow G$ defined by $f(g) = g^n$ for all $g \in G$ is surjective.
 \end{definition}

\begin{theorem}\label{T:4.5}
Let $k$ and $n$ be natural numbers. If $(g\otimes g)^k=1$ for all $g\in G$ and for any odd integer $k$ or there exist an even integer $n$ such that $G$ has $n^{th}$ roots, then in both cases $\pi_{2}^S(K(G,1))\cong H_2 (G)$.
\end{theorem}

\begin{proof}
First we will prove that $\nabla (G) = \Delta (G)$ in both the given cases. By Proposition $\ref{L:2.3}$, $\Delta (G) \subseteq \nabla (G)$. If $k=2n+1$, then $g \otimes g = (g \otimes g)^{-2n}= [(g\otimes g)(g\otimes g)]^{-n}\in \Delta (G)$. Now if $G$ has $n^{th}$ roots for $n=2l$ with $l\in \mathbb{Z}$, then $g = y^{2l}$. Thus $g = x^2$ where $x = y^l$. Then we have,
\begin{align*}
g\otimes g &= x^2\otimes x^2\\
&=(x\otimes x)(x\otimes x)(x\otimes x)(x\otimes x) \hspace{6mm}\mbox{[(\ref{E:1.2.1}),\;(\ref{E:1.2.2})\;and \;Proposition \ref{L:2.1}]}\\
&\subseteq \Delta (G)
\end{align*}
Hence the result follows by commutative diagram \ref{D:1}.
\end{proof}

\begin{corollary}
 If $G$ has odd exponent, then $\pi_{2}^S(K(G,1))\cong H_2 (G)$.
\end{corollary}

The idea of the proof of Theorem $\ref{T:4.2}$ can be adapted to prove the following corollary which can also be found in \cite{BFM} as Theorem 1. The proof given below is different and it provides  an explicit splitting.

\begin{corollary}\label{C:4.7}
Let $G$ be a group such that its abelianization is finitely generated. If $G'$ has a complement in $G$, then
\begin{itemize}
\item[(i)] $ G\otimes G \cong \nabla (G) \times (G \wedge G)$. In particular, $\pi_3(SK(G,1))\cong H_2(G)\times \nabla (G)$.
\item[(ii)] Let $N$ be a normal subgroup of $G\otimes G$. If $N\subseteq \nabla (G)$, then $(G\otimes G)/N\cong (\nabla (G))/N\times (G\wedge G)$.
\end{itemize}
\end{corollary}

\begin{proof}
\begin{itemize}

\item[(i)]
By Proposition $\ref{P:2.4}$, we obtain that $\nabla (G)= \nabla (G_{ab})$. Hence we have the following commutative diagram where $j$ and $i$ are corresponding inclusion maps and $p$ is the natural projection.
\begin{equation*}
\xymatrix@+20pt{
1\ \ar@{->}[r]
&\nabla (G)\ar@{->}^{j}[r]
\ar@{}^{||\wr}[d]
 &G\otimes G\ar@{->}[r]
\ar@{->}^p[d]
&G\wedge G \ar@{->}[r]
\ar@{->}[d]
&1 \\
1\ \ar@{->}[r]
&\nabla (G_{ab})\ar@{->}^{i}[r]
 &G_{ab}\otimes G_{ab}\ar@{->}[r]
&G_{ab}\wedge G_{ab}\ar@{->}[r]
&1.
}\end{equation*}
 Our aim is to show that the top row of the above commutative diagram splits.  By Lemma \ref{L:4.1}$(i)$, the bottom row splits and let
  $f: G_{ab}\otimes G_{ab} \longrightarrow \nabla (G_{ab})$  be the splitting map.  Set $f'=1\circ f\circ p$. It is easily checked that $f'$ is the required splitting. 
 \item[(ii)] Proceeding the same way as in $(i)$ above and using $(ii)$ of Lemma \ref{L:4.1} instead of $(i)$ finishes the proof.
\end{itemize}
\end{proof}

\begin{lemma}\label{L:4.9}
Let $A$ be an abelian torsion group with no elements of order 2, then $A\otimes A \cong \nabla A \times (A\wedge A)$.
\end{lemma}

\begin{proof}
Consider the short exact sequence, $1\to \nabla(A) \to A\otimes A \to A\wedge A \to 1$ where $\alpha$ is the inclusion map from $\nabla (A)$ to $A\otimes A$. Define $f : A\otimes A \to \nabla (A)$ as $f(x\otimes y) = (x^{\frac{1}{2}}\otimes y)(y\otimes x^{\frac{1}{2}})$. We claim that $f$ is a well defined homomorphism and in fact the required splitting map. Towards that end, note that for $x \in A$ there exist an odd integer $k$ such that $x^k = 1$. By Bezout's identity,  we have $x = x^{km+2n} = (x^n)^2$. Therefore $A$ has square roots which are also unique since $A$ has no elements of order $2$. Now it can be easily verified that $f$ is a well defined homomorphism. For $ a = x^2 \in A$ observe that,
 \begin{align*}
 f\circ\alpha(a\otimes a) &=  (x\otimes a)(a\otimes x)\\
 &= (x\otimes x^2)(x^2\otimes x)\\
 &= (x^2\otimes x^2)\\
 &= (a\otimes a)
 \end{align*}
So $f\circ\alpha = 1_{\nabla A}$ and the result follows.
\end{proof}

\begin{theorem}\label{T:4.9}
If $G$ is a torsion group with no elements of order $2$, then $G\otimes G \cong \nabla(G)\times (G\wedge G)$. In particular, $\pi_2^S(K(G,1))\cong H_2(G)$ and $\pi_3(SK(G,1))\cong H_2(G)\times \nabla(G)$. 
\end{theorem}
\begin{proof}
First we will prove that $\nabla (G)\cong \nabla (G_{ab})$. Recall that $G = \varinjlim G_i$, where each $G_i$ is a finitely generated subgroup of $G$. By Theorem 1.3 of \cite{BFM}, we have $\nabla(G_i) \cong \nabla(G_i)_{ab}$. Taking direct limit on both sides and noting that direct limit commutes with both nabla and abelianization, we obtain $\nabla (G)\cong \nabla(\varinjlim (G_i)_{ab})\cong \nabla (G_{ab})$. By Lemma \ref{L:4.9}, the exact sequence $1\to \nabla(G_{ab}) \to G_{ab}\otimes G_{ab} \to G_{ab}\wedge G_{ab} \to 1$ splits. Now proceeding as in Corollary \ref{C:4.7} the result follows.
\end{proof}

\begin{Remark}
Similarly it can be proved that $\nabla'(G) \cong \nabla'(G_{ab})$. If we can obtain the splitting of the short exact sequence $1\to \nabla'(G_{ab}) \to G_{ab}\widetilde{\otimes} G_{ab} \to G_{ab}\widetilde{\wedge} G_{ab}\to 1$, then using the strategy described in the beginning of this section, we can obtain Theorem \ref{T:4.2} without the hypothesis that $G$ is finitely generated.
\end{Remark}

Now we prove the following structural result for $G\otimes G$ by constructing an explicit splitting map.

\begin{theorem}\label{T:4.8} Let $k\in \mathbb{N}$ be an odd integer. If $(g\otimes g)^k=1$ for all $g\in G$, then $G\otimes G \cong \nabla (G) \times G \wedge G $. In particular, if the exponent of $G$ is odd, then $\pi_3(SK(G,1))\cong H_2(G)\times \nabla (G)\cong \pi_2^S(K(G,1))\times \nabla (G)$.
\end{theorem}
\begin{proof}
Consider the  short exact sequence $1\to \nabla(G)  \to G\otimes G\to G \wedge G \to 1$, where $\alpha$ is the natural inclusion from $\nabla (G)$ to $G\otimes G$. Since $k$ is odd, $k = 2n+1$, $n\in \mathbb{N}$. Define  a splitting map $\alpha ' : G \otimes G \to \nabla (G)$ by $\alpha'(g \otimes h) = [(g \otimes h)(h \otimes g)]^{-n}$. We will prove that $\alpha'$ is a well-defined homomorphism. Towards that end,
\begin{align*}
\alpha'(gg_1 \otimes h) &= [(gg_1 \otimes h)(h \otimes gg_1)]^{-n} \\ &= [^g(g_1 \otimes h)(g \otimes h)(h \otimes g) ^g(h \otimes g_1)]^{-n} \\
 &=  [(g \otimes h)(h \otimes g) ^g\{(g_1 \otimes h)(h \otimes g_1)\}]^{-n} \\ &=
 [(g \otimes h)(h \otimes g)(g_1 \otimes h)(h \otimes g_1)]^{-n}\\
\mbox{and               } \\
 \alpha'( ^g g_1 \otimes ^g h) &= [( ^g g_1 \otimes ^g h)( ^g h \otimes ^g g_1)]^{-n}\\ &= [ ^g\{(g_1 \otimes h)(h \otimes g_1)\}]^{-n}\\
  &= [(g_1 \otimes h)(h \otimes g_1)]^{-n}.
 \end{align*}
 Therefore $\alpha'(gg_1 \otimes h) = \alpha'( ^g g_1 \otimes ^g h) \alpha( g \otimes h).$ Similarly $\alpha'(g \otimes hh_1) =
 \alpha'(g \otimes h) \alpha'(^h (g \otimes h_1))$. Hence $\alpha'$ is a homomorphism.
Note that $\alpha' \alpha (g \otimes g) = \alpha '(g \otimes g) = \{(g \otimes g)(g \otimes g)\}^{-n} =
 (g \otimes g)^{-2n} = (g \otimes g)$. Therefore $\alpha'\alpha = 1_{\nabla (G)}$ and the result follows.
 The last claim now follows as a consequence of Theorem \ref{T:4.5}.
\end{proof}

\begin{corollary}\label{C:4.14}
If $G$ is a group of odd order, then $\pi_{2}^S(K(G,1))=H_2 (G)$ and $\pi_{3}(SK(G,1))\cong \nabla (G) \times \pi_{2}^S(K(G,1))=\nabla (G) \times H_2 (G)$ .
 \end{corollary}

\begin{corollary}\label{C:4.15}
If $\Delta(G)$ has unique $n^{th}$ roots where $n$ is even then, $\pi_{2}^S(K(G,1))=H_2 (G)$ and $\pi_{3}(SK(G,1))\cong \nabla (G) \times \pi_{2}^S(K(G,1))=\nabla (G) \times H_2 (G)$ .
\end{corollary}

\begin{corollary}\label{C:4.16}
Let $G$ be a perfect group, then $\pi_3(SK(G,1))=\pi_{2}^S(K(G,1))=H_2 (G)$.
\end{corollary}

\section{Applications}

In this section we compute  $\pi_3(SK(G,1))$ and $\pi_2^S(K(G,1))$ for symmetric groups, alternating groups and general linear groups.
\begin{corollary}
Let $A_n$ denote the alternating group on $n$ letters, then
\begin{align*}
\pi_3({SK(A_n,1})&=\pi_{2}^S({K(A_n,1)})\cong \mathbb{Z}_2 \hspace{5mm} \mbox{if}\ n\geq 5,\ n\neq 6,7\\
\pi_3({SK(A_n,1)})&=\pi_{2}^S({K(A_n,1)})\cong \mathbb{Z}_6 \hspace{5mm} \mbox{if}\ n = 6,7
\end{align*}
\end{corollary}
\begin{proof}  This follows from Corollary $\ref{C:4.16}$.
\end{proof}
\begin{prop}
Let $S_n$ denote the symmetric group on $n$ letters, then
\begin{itemize}
\item[(i)]\begin{align*}
\pi_3({SK(S_n,1)})&\cong \mathbb{Z}_2\times \mathbb{Z}_2 \hspace{6mm} \mbox{when}\ n\geq 4\\
\pi_3({SK(S_n,1)})&\cong \mathbb{Z}_2 \hspace{15mm} \mbox{otherwise}\\
\end{align*}
\item[(ii)]\begin{align*}
\pi^S_2({SK(S_n,1)})&\cong \mathbb{Z}_2\times \mathbb{Z}_2 \hspace{6mm} \mbox{when}\ n\geq 4\\
\pi^S_2({SK(S_n,1)})&\cong \mathbb{Z}_2 \hspace{15mm} \mbox{otherwise}\\
\end{align*}
\end{itemize}
\end{prop}
\begin{proof}
\begin{itemize}
\item[(i)]We have the following short exact sequence $1\to A_n\to S_n\to \mathbb{Z}_2\to 1$ which splits. Recall that $H_2(S_n)\cong \mathbb{Z}_2$ when $n\geq 4$ and $0$ otherwise. Applying Corollary \ref{C:4.7} and noting that $\nabla(\mathbb{Z}_2)\cong \mathbb{Z}_2$, it follows that $\pi_3({SK(S_n,1)})\cong \mathbb{Z}_2\times \mathbb{Z}_2$ when $n\geq 4$ and $\pi_3({SK(S_n,1)})\cong \mathbb{Z}_2$ otherwise.
\item[(ii)] Note that $\pi_{2}^S({K(S_n,1)})=\pi_3({SK(S_n,1)})/{\Delta(G)}$. By Proposition \ref{P:2.4}, we have that  $\Delta(G)=\Delta(G_{ab})$. Noting that $\Delta(\mathbb{Z}_2)=0$, it follows that $\pi_{2}^S({K(S_n,1)})\cong \mathbb{Z}_2\times \mathbb{Z}_2$ when $n\geq 4$ and is isomorphic to $\mathbb{Z}_2$ otherwise.
\end{itemize}
\end{proof}

\begin{prop}
Let $\mathbb{F}_p$ denote the finite field of characteristic $p$ then,\newline
\begin{itemize}

\item[(i)]
For $p > 3$,
\begin{align*}
&\pi_3({SK(GL(n, \mathbb{F}_p),1}))\cong\mathbb{Z}_{p-1}
\end{align*}
For $p = 3$,
\begin{align*}
&\pi_3({SK(GL(n, \mathbb{F}_p),1}))\cong\mathbb{Z}_{2}\hspace{8mm}\mbox{if}\ n \neq 2\\
\end{align*}
For $p = 2$,
\begin{align*}
&\pi_3({SK(GL(n, \mathbb{F}_p),1}))\cong\mathbb{Z}_2 \hspace{8mm}\mbox{if}\ n \in \{3,4\} \\
&\pi_3({SK(GL(n, \mathbb{F}_p),1}))\ \mbox{is\ trivial\ for}\ n > 4\\
\end{align*}
\item[(ii)]
For $p > 3$,
\begin{align*}
&\pi^S_2({K(GL(n, \mathbb{F}_p),1}))\cong\mathbb{Z}_{2}
\end{align*}
For $p = 3$,
\begin{align*}
&\pi^S_2({K(GL(n, \mathbb{F}_p),1}))\cong\mathbb{Z}_{2}\hspace{8mm}\mbox{if}\ n \neq 2\\
\end{align*}
For $p = 2$,
\begin{align*}
&\pi^S_2({K(GL(n, \mathbb{F}_p),1}))\cong \mathbb{Z}_{2} \hspace{8mm}\mbox{if}\ n \in \{3,4\} \\
&\pi^S_2({K(GL(n, \mathbb{F}_p),1}))\ \mbox{is\ trivial\ for}\ n > 4\\
\end{align*}

\end{itemize}

\end{prop}
\begin{proof}
\begin{itemize}
\item[(i)]
We know that $[GL(n, \mathbb{F}_p),GL(n, \mathbb{F}_p)]= SL(n, \mathbb{F}_p)$ unless $n=2$ and $p=2,3$. We have the following short exact sequence which splits, $1\to SL(n, \mathbb{F}_p)\to GL(n, \mathbb{F}_p)\to  {\mathbb{F}_p}^*\to 1$. The Schur multiplier of $GL(n, \mathbb{F}_p)$ is trivial for $(n,p)\notin \{(3,2), (4,2)\}$ and $\mathbb{Z}_2$ otherwise. Hence $\pi_3({SK(GL(n, \mathbb{F}_p),1}))\cong\mathbb{Z}_{p-1}$ for $p= 3,n\neq 2$ or $p > 3$ and is isomorphic to $\mathbb{Z}_2$ if $(n,p)\in \{(3,2),(4,2)\}$.
\item[(ii)] After noting that $\Delta(\mathbb{Z}_{p-1})\cong \mathbb{Z}_{\frac{p-1}{2}} $ for $p\neq 2$, the proof follows the proof of $(ii)$ of the previous proposition mutatis mutandis.
\end{itemize}
\end{proof}

In Example 3.7 of \cite{BL2}, the authors compute $\pi_3(SK(D_{2n}, 1))$. In the next corollary, we do the same for $\pi_2^S(K(D_{2n},1))$.

\begin{corollary} 
Let $D_{2n}$ be the Dihedral group of order $2n$ and $Q_8$ be the  Quaternion group, then
 \begin{align*}
 &\pi_2^S(K(D_{2n},1)) = \mathbb{Z}_2,\ \mbox{if\ n\ is\ odd}\\
 &\hspace{2.6cm}= \mathbb{Z}_2^3,\ \mbox{if\ n\ is\ even}\\
 &\pi_2^S(K(Q_8,1)) = \mathbb{Z}_2^2
 \end{align*}
\end{corollary}
\begin{proof}
 It is well known that $(D_{2n})_{ab}\cong \mathbb{Z}_2$ when $n$ is odd, $(D_{2n})_{ab}\cong \mathbb{Z}_2\times\mathbb{Z}_2$ when $n$ is even and $(Q_8)_{ab} \cong \mathbb{Z}_2\times \mathbb{Z}_2$. Now using Theorem \ref{T:4.2} the result follows.
\end{proof}

Denoting the infinite general linear group with coefficients in a commutative ring $R$ as $GL(R)$, we have the following proposition.

\begin{prop} Let $R$ be a commutative ring with unit. If $R$ is a field,  Euclidean domain, semi-local ring or ring of integers of a number field, then $\pi_3(SK(GL(R),1))\cong H_2(GL(R))\times \nabla (K_1(R))$. In particular, $\pi_3(SK(GL(\mathbb{Z}),1))\\ \cong \mathbb{Z}_2\times \mathbb{Z}_2$.
\end{prop}
\begin{proof}
First recall (\cite{CW}) that the special Whitehead group $SK_1(R)$ is trivial when $R$ is a Euclidean domain, semi-local ring or the ring of integers of a number field. Hence we have a split exact sequence $1\to E(R)\to GL(R)\to K_1(R)\to 1$. Thus $\pi_3(SK(GL(R),1))\cong \nabla(GL(R)_{ab})\times H_2(GL(R))\cong \nabla(K_1(R))\times H_2(GL(R))$. The second part follows after noting that $K_1(\mathbb{Z})\cong \mathbb{Z}_2$ and $H_2(GL(\mathbb{Z}))\cong \mathbb{Z}_2$.
\end{proof}

The next corollary can be derived from the above proposition, but instead we will put to use Proposition \ref{P:3.3}.

\begin{corollary}
Let $F_q$ be a finite field of $p^n$ elements where $p$ is a prime. Then
\begin{itemize}
\item[(i)]
\begin{equation*}
\pi_3(SK(GL(F_q),1))\cong \mathbb{Z}_{q-1}
\end{equation*}
\item[(ii)] 
\begin{align*}
&\pi_2^S(K(GL(F_q),1))\cong \mathbb{Z}_2\hspace{6mm} \mbox{when\ p\ is\ odd}\\
&\pi_2^S(K(GL(F_q),1)) \mbox{\ is\ trivial\ when\ p = 2}
\end{align*}

\end{itemize}
\end{corollary}
\begin{proof}
\begin{itemize}
\item[(i)]
Noting that Milnor's $K_2(F_q)$ is trivial (Corollary to Matsumoto's theorem, \cite{CW}) and $E(R)$ is perfect, it follows by Proposition \ref{P:3.3} that $\pi_3(SK(GL(F_q),1)) \cong \pi_3(SK(F_q^*,1))$. Since $F_q^*$ is abelian, we obtain that $\pi_3(SK(F_q^*,1)) \cong F_q^*\otimes F_q^* \cong  Z_{q-1}$.
\item[(ii)]
Again by Proposition \ref{P:3.3}, $\pi_2^S(K(GL(F_q),1)) \cong \pi_2^S(K(F_q^*,1))$ and since $F_q^*$ is abelian, $\pi_2^S(K(F_q^*,1))\cong (F_q^*\otimes F_q^*)/\Delta(F_q^*)$. When $p$ is odd, $\Delta(F_q^*) \cong \mathbb{Z}_{\frac{q-1}{2}}$ and $\pi_2^S(K(GL(F_q),1)) \cong \mathbb{Z}_2$. If $p = 2$, $\Delta(F_q^*) \cong \nabla(F_q^*) \cong F_q^*$ and $\pi_2^S(K(GL(F_q),1))$ is trivial.
\end{itemize}
\end{proof}

Using the bound for $p$ groups given by Green in \cite{JG} and Theorem 4, Chapter IX of \cite{S}, it follows that if $G$ is a finite group with order $|G|=p_1^{a_1}\cdots p_n^{a_n}$, then $|H_2(G)|\leq \prod p_i^{\frac{a_i(a_i-1)}{2}}$. Using this bound and Theorem \ref{T:4.2}, the next corollary follows easily.

\begin{corollary}
Let $G$ be a finite group. If $|G|=p_1^{a_1}\cdots p_n^{a_n}$, then $|\pi_2^S(K(G,1))|\leq 2^{r+k} \times \prod p_i^{\frac{a_i(a_i-1)}{2}}$.
\end{corollary}

\end{document}